\documentclass{amsart}

\usepackage{amssymb, amsmath, amsthm}

\usepackage[english]{babel}

\author{Rafael Torres}

\title{Irreducible 4-Manifolds with Abelian Non-cyclic Fundamental
Group of Small Rank}

\address{Max-Planck Institut f\"ur Mathematik\\Vivatsgasse 7\\53111\\Bonn, Germany}

\email{torres@mpim-bonn.mpg.de}

\date{June 24, 2009. We thank S. Baldridge, J. Hillman, P. Kirk and M. Kreck for their useful comments on an earlier version of the paper. We thank the Max-Planck Institut f\"ur Mathematik in Bonn for its
warm hospitality and amazing working conditions. This work was realized under an IMPRS scholarship from the Max-Planck Society.}

\theoremstyle{plain}

\newtheorem{theorem}[equation]{Theorem}

\newtheorem{proposition}[equation]{Proposition}

\newtheorem{lemma}[equation]{Lemma}

\newtheorem{remark}{Remark}

\theoremstyle{definition}

\newtheorem{definition}[equation]{Definition}

\newcommand{\R}{\mathbb{R}}

\newcommand{\Z}{\mathbb{Z}}

\begin{document}

\maketitle

In this paper we construct several irreducible 4-manifolds, both small and arbitrarily large, with abelian non-cyclic fundamental group. The manufacturing procedure
allows us to fill in numerous points in the geography plane of symplectic manifolds with $\pi_1 = \Z \oplus \Z, \Z\oplus \Z_p$ and $\Z_q \oplus \Z_p$ ($gcd(p, q) \neq
1$). We then study the botany of these points for $\pi_1 = \Z_p \oplus \Z_p$.\\

\section{Manufactured Manifolds}

The main results in this paper are:

\begin{theorem} Let $G$ be either $\Z \oplus \Z$, $\Z \oplus \Z_p$ or $\Z_q \oplus \Z_p$. Let $n\geq 1$ and $m\geq 1$. For each of the following pairs of
integers
\begin{enumerate}
\item $(c, \chi) = (7n, n)$,
\item $(c, \chi) = (5n, n)$,
\item $(c, \chi) = (4n, n)$,
\item $(c, \chi) = (2n, n)$,
\item $(c, \chi) = ((6 + 8g)n, (1 + g)n$ (for $g\geq 0$),
\item $(c, \chi) = (7n + (6 + 8g)m, n + (1 + g)m)$,
\item $(c, \chi) = (7n + 5m, n + m)$,
\item $(c, \chi) = (7n + 4m, n + m)$,
\item $(c, \chi) = (7n + 2m, n + m)$,
\item $(c, \chi) = ((6 + 8g)n + 5m,(1 + g)n + m)$ (for $g\geq 0$),
\item $(c, \chi) = ((6 + 8g)n + 4m,(1 + g)n + m)$ (for $g\geq 0$),
\item $(c, \chi) = ((6 + 8g)n + 2m,(1 + g)n + m)$ (for $g\geq 0$),
\item $(c, \chi) = (5n + 4m , n + m)$,
\item $(c, \chi) = (5n + 2m , n + m)$,
\item $(c, \chi) = (4n + 2m , n + m)$ and
\end{enumerate}
there exists a symplectic irreducible 4-manifold $X$ with 
\begin{center} 
$\pi_1(X) = G$ and $(c_1^2(X), \chi_h(X)) = (c, \chi)$.\\
\end{center}
\end{theorem}

\begin{proposition} Fix $\pi_1(X) = \Z_p \oplus \Z_p$, where $p$ is a prime number greater than two. Let $(c, \chi)$ be any pair of 
integers given in Theorem 1 such that $n + m \geq 2$. There
exists an infinite family $\{X_n\}$ of homeomorphic, pairwise non-diffeomorphic irreducible smooth non-symplectic 4-manifolds realizing the coordinates $(c,
\chi)$.
\end{proposition}

The characteristic numbers are given in terms of $\chi_h = 1/4(e + \sigma)$ and $c_1^2 = 2 e + 3 \sigma$, where $e$ is the Euler
characteristic of the manifold $X$ and $\sigma$ its signature.\\

The geography problem for abelian fundamental groups of small rank has already been previously studied with great success. In Gompf's
gorgeous paper \cite{[Go]} where the symplectic sum operation was introduced, infinitely
many minimal symplectic 4-manifolds with $b_2^+ > 1$ were constructed. Gompf also constructed a new family of symplectic spin
4-manifolds with any prescribed fundamental group. In \cite{[BK3]}, \cite{[BK4]} and \cite{[BK5]}, more and smaller symplectic manifolds were 
constructed.\\

Other construction techniques have also been implemented. For the group $\pi_1 = \Z \oplus \Z_p$, examples with big Euler characteristic where constructed using 
genus 2 Lefschetz fibrations in \cite{[OS]} and \cite{[IS]}. Results studying the symplectic geography for prescribed fundamental groups appeared in
\cite{[BK5]} and \cite{[BK3]}. Concerning the botany, J. Park in \cite{[JP]} constructed infinitely many smooth structures on big
4-manifolds with finitely generated fundamental group.\\

The addition of Luttinger surgery (cf. \cite{[Lu]}, \cite{[ADK]}) into the manufacturing procedure has provided clean constructions to study rather
effectively the geography of simply connected 4-manifolds (cf. \cite{[BK3]}, \cite{[ABBKP]}, \cite{[AP]}). On the botany part, the technique of doing
of using a nullhomologous torus as a dial in order to change the smooth structure developed in \cite{[FS]} and \cite{[FPS]} has proven succesful 
to study the botany. In this paper, we apply these efforts to manifolds with the three given fundamental groups.\\  

Our results provide manifolds with both $12\chi - c$ small and arbitrarily large. Most of the points filled in by Theorem 1 were not yet
considered elsewhere. For example, the point $(7, 1)$ corresponds to the 
smallest manifold built up to now. A blunt overlap occurs for the points $(6 + 8g, 1 + g)$, $(5, 1)$ and $(4,
1)$, which have been filled in already by constructions given in \cite{[BK3]} and \cite{[BK4]}; we are using their constructions to
build larger manifolds, thus filling in considerably many more points. The existence of at least two smooth structures on complex surfaces with finite non-cyclic fundamental groups was first studied in \cite{[HK93]}. Proposition 2 takes advantage of the
recent techniques and offers a myriad of new exotic irreducible 4-manifolds with finite abelian, yet non-cyclic fundamental group hosting infinitely many 
smooth structures; it includes the smallest manifold with such $\pi_1$ known to posses this quality.\\

The assumption $gcd(p, q) \neq 1$ serves the sole purpose of emphasizing that the results in this paper are disjoint from the cyclic case studied in
\cite{[To]}. We feel the results presented here deserve their own space and they should not be buried in a long paper for several reasons.
Amongst them is the employment of the homeomorphism criteria for finite groups of odd order (cf. \cite{[HK93]}) given in Section 6.3.\\




The blueprint of the paper is as follows. The geography is addressed first; Section 2 starts by describing the ingredients we will use to build the
manifolds of Theorem 1. The manufacturing procedure starts later
on in this section. The results that allow us to conclude irreducibility are presented in Section 3. The fourth section takes care of
the fundamental group calculations. The fifth section gathers up our efforts into the proof of Theorem 1.
The last part of the paper goes into the botany, where Section 6 takes on the existence of the exotic smooth structures claimed in Proposition 2.\\

\section{Raw Materials}


The following definition was introduced in \cite{[ABBKP]}.

\begin{definition} An ordered triple $(X, T_1, T_2)$ consisting of a symplectic 4-manifold $X$ and two disjointly 
embedded Lagrangian tori $T_1$ and $T_2$ is called a \emph{telescoping triple} if 
\begin{enumerate}
\item The tori $T_1$ and $T_2$ span a 2-dimensional subspace of $H_2(X; \R)$.
\item $\pi_1(X) \cong \Z^2$ and the inclusion induces an isomorphism $\pi_1(X - (T_1 \cup T_2)) \rightarrow \pi_1(X)$. In particular, the meridians 
of the tori are trivial in $\pi_1(X - (T_1 \cup T_2))$.
\item The image of the homomorphism induced by the corresponding inclusion $\pi_1(T_1) \rightarrow \pi_1(X)$ is a summand $\Z \subset \pi_1(X)$.
\item The homomorphism induced by inclusion $\pi_1(T_2) \rightarrow \pi_1(X)$ is an isomorphism.
\end{enumerate}

\end{definition}

The telescoping triple is called \emph{minimal} if $X$ itself is minimal. Notice the importance of the order 
of the tori. The meridians $\mu_{T_1}$, $\mu_{T_2}$ in $\pi_1(X - (T_1 \cup T_2))$ are trivial and the 
relevant fundamental groups are abelian. The push off of an oriented loop $\gamma \subset T_i$ into $X - (T_1 \cup T_2)$ with 
respect to any (Lagrangian) framing of the normal bundle of $T_i$ represents a well defined element of 
$\pi_1(X - (T_1 \cup T_2) )$ which is independent of the choices of framing and base-point.\\

The first condition assures us that the Lagrangian tori $T_1$ and $T_2$ are linearly independent in $H_2(X; \R)$. This allows 
for the symplectic form on $X$  to be slightly perturbed so that one of the $T_i$ remains Lagrangian while the other becomes 
symplectic. The symplectic form can also be perturbed in such way that both tori become symplectic. If we were to consider a symplectic 
surface $F$ in $X$ disjoint from $T_1$ and $T_2$, the perturbed symplectic form can be chosen so that $F$ remains symplectic.\\

Removing a surface from a 4-manifold usually introduces new generators into the fundamental group of the resulting manifold. 
The second condition indicates that the meridians are nullhomotopic in the complement and, thus, the fundamental group of the 
manifold and the fundamental group of the complement of the tori in the manifold coincide.\\

Out of two telescoping triples, one is able to produce another telescoping triple as follows. If both $X$ and $X'$ are symplectic manifolds, then the symplectic sum along the symplectic tori $X \#_{T_2, T_1'} X'$ has a
symplectic structure (\cite{[Go]}). If both $X$ and $X'$ are minimal, then the resulting telescoping triple
is minimal too (by Usher's theorem cf. \cite{[Ush]}). \\

\begin{proposition} (cf. \cite{[ABBKP]}). Let $(X, T_1, T_2)$ and $(X', T_1', T_2')$ be two telescoping triples. Then for an appropriate 
gluing map the triple

\begin{center}
$(X \#_{T_2, T_1'} X', T_1, T_2')$
\end{center}
is again a telescoping triple.\\
The Euler characteristic and the signature of $X \#_{T_2, T_1'} X'$ are given by $e(X) + e(X')$ and $\sigma(X) + \sigma(X')$.
\end{proposition}


We refer the reader to theorems 20 and 13 and to proposition 12 in \cite{[BK3]} for the proof and for more details. The building blocks we
will used are gathered together in the following theorem. \\

\begin{theorem} 
\begin{itemize} 
\item There exists a minimal telescoping triple $(A, T_1, T_2)$ with $e(A) = 5$, $\sigma(A) = - 1$.
\item For each $g\geq 0$, there exists a minimal telescoping triple $(B_g, T_1, T_2)$ satisfying $e(B_g) = 6 +
4g$, $\sigma(B_g) = - 2$.
\item There exists a minimal telescoping triple $(C, T_1, T_2)$ with $e(C) = 7$, $\sigma(C) = - 3$.
\item There exists a minimal telescoping triple $(D, T_1, T_2)$ with $e(D) = 8$, $\sigma(D)
= - 4$.
\item There exists a minimal telescoping triple $(F, T_1, T_2)$ with $e(F) = 10$, $\sigma(F) = - 6$.\\ 
\end{itemize}
\end{theorem}

The manifolds $B_g$, $D$ and $F$ were already built in \cite{[ABBKP]}. They are taken out of the
constructions given in \cite{[BK3]} by the following mechanism. The main goal of \cite{[BK3]} is to construct simply
connected 4-manifolds by applying Luttinger surgery to symplectic sums. If one is careful about the fundamental
fundamental group calculations, the procedure can be interrumpted by NOT performing two surgeries, and thus obtain a
symplectic manifold with $\pi_1 = \Z \oplus \Z$. Furthermore, the skipped surgeries have to be chosen carefully 
so that the unused Lagrangian tori comply with the requirements and the pieces can then be aligned into a telescoping
triple.\\

To finish the proof of Theorem 5, we construct $(A, T_1, T_2)$ and $(C, T_1, T_2)$ by applying this mechanism to the
constructions in \cite{[AP]}. This is done in the following two lemmas, where we follow the notation of \cite{[AP]}.\\

\begin{lemma} There exists a telescoping triple $(A, T_1, T_2)$ with $e(C) = 5$ and $\sigma(C) = -1$.\\
\end{lemma}

\begin{proof} This telescoping triple is obtained out of the construction of an exotic irreducible symplectic $\mathbb{CP}^2 \# 2\overline{\mathbb{CP}}^2$ given in \cite{[AP]}. 
The two surgeries to be skipped are $(a_2' \times c', c', +1/p)$ and $(b_1'\times c'', b_1', -1)$ (the notation is explained in \cite{[FS]}). Rename the corresponding tori $T_1$ and $T_2$. This procedure manufactures a minimal symplectic
manifold $A$. Notice that the tori are linearly independent in $H_2(A; \R)$. We need to check that such manifold has indeed $\pi_1 = \Z^2$ and that it contains
the required tori.\\

Let us begin with the fundamental group calculations. By combining the relations coming from the surgeries $(a_1'\times c',
a_1', -1)$ and $(a_2'' \times d', d', +1)$ that where performed on the $\Sigma_2 \times T^2$ block (see \cite{[AP]} for details) we have
$\alpha_1 = a_1 = [b_1^{-1}, d^{-1}] = [b_1^{-1}, [b_2, c^{-1}]^{-1}] = [b_1^{-1}, [c^{-1}, b_2]] = 1$. The last commutator is trivial
since $b_1$ commutes with both $c_1$ and $b_2$. Substituting this in the relations coming from the surgeries applied to the
building block $T^4\# \overline{\mathbb{CP}}^2$, we obtain $\alpha_3 = a_2 = 1$ and $\alpha_4 = b_2 = 1$. By looking at the
relations from the other building block we see $d = 1$. Note that the meridians of the surfaces along which the gluing is
performed are trivial. Thus only two commuting generators survive in the group presentation.\\

We check that the meridian of the first torus is $\mu_{T_1} = [d^{-1}, b_2^{-1}] = 1$ and its Lagrangian push-offs are $m_{T_1} = c$ and $l_{T_1} =
a_2 = 1$. For the torus $T_2$ one sees $\mu_{T_2} = [a_1^{-1}, d] = 1$ and its Lagrangian push-offs are $m_{T_2} = c$ and $l_{T_2} = b_1$.
So, $\pi_1(A - (T_1 \cup T_2))$ is generated by the commuting elements $b_1$ and $c$. By a Mayer-Vietoris sequence we see $H_1(A - (T_1 \cup T_2)) =
\Z^2$. Thus $\pi_1(A - (T_1\cup T_2) = \Z b_1\oplus \Z c$. We conclude $(A, T_1, T_2)$ is a telescoping triple.\\
\end{proof}

\begin{lemma} There exists a telescoping triple $(C, T_1, T_2)$ with $e(C) = 7$ and $\sigma(C) = -3$.
\end{lemma}

\begin{proof} We follow the construction of an exotic irreducible symplectic $\mathbb{CP}^2 \# 4\overline{\mathbb{CP}}^2$ given in \cite{[AP]}. The
surgeries $(\alpha'_2 \times \alpha_3'', \alpha_2', -1)$ in the $T^4 \# 2
\overline{\mathbb{CP}}^2$ block and $(\alpha_2'' \times \alpha_4', \alpha_4', -1)$ in the $T^4 \# 
\overline{\mathbb{CP}}^2$ block will NOT be performed. Call these tori $T_2$ and $T_1$ respectively and the resulting manifold $C$. 
Notice that they are linearly independent in $H_2(C; \R)$.\\

We apply $(\alpha'_1 \times \alpha_3', \alpha_1', -1)$ on the $T^4 \# 2\overline{\mathbb{CP}}^2$. This 
introduces the relation $\alpha_1 = [\alpha_2^{-1}, \alpha_4^{-1}]$. Using the commutator $[\alpha_2, \alpha_4] = 1$, one sees $\alpha_1 = 1$. The
relation $\alpha_3 = [\alpha_1^{-1}, \alpha_4^{-1}]$ obtained by applying a Luttinger surgery on the $T^4\# \overline{\mathbb{CP}}^2$ building block
implies $\alpha_3 = 1$. The surfaces of genus 2 along which the symplectic sum is performed have trivial meridians.\\

The meridian of $T_1$ is $\mu_{T_1} = [a_1^{-1}, \alpha_4] = 1$ and its Lagrangian push-offs are $m_{T_1} = \alpha_2$ and $l_{T_1} = \alpha_3 = 1$. 
The meridian of $T_2$ is given by $\mu_{T_2} = [\alpha_1, \alpha_3^{-1}] = 1$ and its Lagrangian push-offs are $m_{T_2} = \alpha_4$ and $l_{T_2} =
\alpha_2$. We have that $\pi_1(C - (T_1 \cup T_2))$ is generated by the commuting elements $\alpha_2$ and $\alpha_4$. The Mayer-Vietoris
sequence computes $H_1(C - (T_1 \cup T_2)) = \Z^2$, thus $\pi_1(C - (T_1 \cup T_2)) = \Z\alpha_2 \oplus \Z\alpha_4$. Thus, $(C, T_1, T_2)$ is a telescoping triple.
\end{proof}

\begin{remark} One is able to realize the point $(c^2_1, \chi_h) = (3, 1)$ for the fundamental groups $\pi_1 = \Z^2$ and $\pi_1 = \Z$ during the manufacturing process of an exotic irreducible symplectic $\mathbb{CP}^2
\# 6\overline{\mathbb{CP}}^2$. Consider the symplectic sum of $T^4 \# \overline{\mathbb{CP}}^2$ and $T^2 \times S^2 \# 4\overline{\mathbb{CP}}^2$ along
a genus 2 surface given in \cite{[AP]}. The resulting minimal symplectic 4-manifold has a fundamental group with the following presentation

\begin{center}
$<\alpha_1, \alpha_2, \alpha_3| [\alpha_1, \alpha_2] = 1, [\alpha_2, \alpha_3] = 1, 
\alpha_1^{-1} = \alpha_3^2> \cong \Z \oplus \Z$.\\
\end{center}

If we apply the surgery $(\alpha_2'' \times \alpha_4', \alpha_4', -1)$, the relation $\alpha_4 = [\alpha_1, \alpha_3^{-1}]$ is introduced 
to the fundamental group presentation and we obtain a manifold with
fundamental group

\begin{center}
$\pi_1 = <\alpha_1, \alpha_3| \alpha_1^{-1} = \alpha_3^2> \cong \Z$.\\
\end{center}

If we apply the surgery $(\alpha_2' \times \alpha_3', \alpha_3', -1)$, the relation $\alpha_3 = [\alpha_1^{-1}, \alpha_4^{-1}]$ is 
introduced to the fundamental group presentation and we obtain a manifold with fundamental group $\pi_1 = <\alpha_2> = \Z$.\\

\end{remark}

One can go on and build more telescoping triples out of these five by using Proposition 4. We proceed to do so now. Let us start by
setting some useful notation. Let $(X, T_1, T_2)$ be a telescoping triple. We will denote by $X_n: = \# n(X)$ the manifold obtained by
building the symplectic sum (cf. \cite{[Go]}) of $n$ copies of $X$ along the proper tori.

\begin{proposition} For each $n\geq 1$ and $m\geq 1$, the following minimal telescoping triples with the given Characteristic numbers exist:
\begin{enumerate}
\item $(A_n, T_1, T_2)$ satisfying $e(A_n) = 5 n$ and $\sigma(A_n) = -n$.
\item $(C_n, T_1, T_2)$ satisfying $e(C_n) = 7 n$ and $\sigma(C_n) = -3 n$.
\item $(D_n, T_1, T_2)$ satisfying $e(D_n) = 8 n$ and $\sigma(D_n) = -4 n$.
\item $(F_n, T_1, T_2)$ satisfying $e(F_n) = 10 n$ and $\sigma(F_n) = -6 n$.
\item $(\# n (B_g), T_1, T_2)$ satisfying $e(\# n (B_g)) = (6 + 4g)n$ and $\sigma(\# n(B_g)) = -2 n$.

\item $(A_n \# m(B_g), T_1, T_2)$ satisfying $e(A_n \# m(B_g)) = 5 n + (6 + 4g)m$ and 
$\sigma(A_n \# m(B_g)) = - n  - 2m$.

\item $(A_n \# C_m, T_1, T_2)$ satisfying $e(A_n \# C_m) = 5n + 7m$ and $\sigma(A_n \# C_m) = -n - 3m$.
\item $(A_n \# D_m, T_1, T_2)$ satisfying $e(A_n \# D_m) = 5n + 8m$ and $\sigma(A_n \# D_m) = -n - 4m$.
\item $(A_n \# F_m, T_1, T_2)$ satisfying $e(A_n \# F_m) = 5n + 10m$ and $\sigma(A_n \# F_m) = -n - 6m$.

\item $(\# n (B_g) \# C_m, T_1, T_2)$ satisfying $e(\# n (B_g) \# C_m) = (6 + 4g)n + 7m$ 
and $\sigma(\# n (B_g) \# C_m) = -2n - 3m$.

\item $(\# n(B_g) \# D_m, T_1, T_2)$ satisfying $e(\# n(B_g) \# D_m) =  (6 + 4g) n + 8 m$ and
$\sigma(n(B_g) \# D_m) = -2 n - 4m$.


\item $(\# n(B_g) \# F_m, T_1, T_2)$ satisfying $e(\# n(B_g) \# F_m) =  (6 + 4g) n + 10 m$ and
$\sigma(n(B_g) \# F_m) = -2 n - 6m$.

\item $(C_n \# D_m, T_1, T_2)$ satisfying $e(C_n \# D_m) =  7 n + 8 m$ and 
$\sigma(C_n \# D_m) = -3 n - 4m$.


\item $(C_n \# F_m, T_1, T_2)$ satisfying $e(C_n \# F_m) =  7 n + 10 m$ and 
$\sigma(C_n \# F_m) = -3 n - 6m$.


\item $(D_n \# F_m, T_1, T_2)$ satisfying $e(D_n \# F_m) =  8 n + 10 m$ and 
$\sigma(D_n \# F_m) = -4 n - 6m$.


\end{enumerate}
\end{proposition}

The claim about minimality is proven in the next section.\\

\section{Minimality and Irreducibility}

The following result allows us to conclude the irreducibility of the manufactured minimal 4-manifolds.\\

\begin{theorem} (Hamilton and Kotschick, \cite{[HKo]}). Minimal symplectic 4-manifolds with residually finite fundamental groups are irreducible.\\
\end{theorem}

Finite groups and free groups are well-known examples of residually finite groups. Since the direct products of residually finite groups
are residually finite groups themselves, the previous result implies that all we need to worry about is producing minimal manifolds in order to
conclude on their irreducibility. This endeavor follows from Usher's theorem.\\ 

\begin{theorem} (Usher, \cite{[Ush]}). Let $X = Y\#_{\Sigma \equiv \Sigma}Y'$ be the
symplectic sum where the surfaces have genus greater than zero.
\begin{enumerate}
\item If either $Y - \Sigma$ or $Y' -\Sigma '$ contains an embedded symplectic
sphere of square -1, then $X$ is not minimal.
\item If one of the summands, say $Y$ for definiteness, admits the structure of
an $S^2$-bundle over a surface of genus $g$ such that $\Sigma$ is a section of
this $S^2$-bundle, then $X$ is minimal if and only if $Y'$ is minimal.
\item In all other cases, $X$ is minimal.\\
\end{enumerate}
\end{theorem}

This theorem implies that the manifolds of Proposition 8 are minimal.\\


\section{Luttinger Surgery and its Effects on $\pi_1$}

Let $T$ be a Lagrangian torus inside a symplectic 4-manifold $M$. \emph{Luttinger surgery} (cf. \cite{[Lu]}, \cite{[ADK]}) is the surgical procedure of
taking out a tubular neighborhood of the torus nbh(T) in $M$ and gluing it back in, in such way that the resulting manifold admits a symplectic
structure. The symplectic form is unchanged away from a neighborhood of $T$. We proceed to give an overview of the process before we get into
the fundamental group calculations.\\

The Darboux-Weinstein theorem (cf. \cite{[MS]}) implies the existence of a parametrization of a tubular neighborhood $T\times D^2 \rightarrow
nbh(T) \subset M$ such that the image of $T\times \{d\}$ is Lagrangian for all $d\in D^2$. Let $d \in D - \{0\}$. The parametrization of the
tubular neighborhood provides us with a particular type of push-off $F_d: T\times \{d\} \subset M - T$ called the \emph{Lagrangian
push-off} or \emph{Lagrangian framing}. Let $\gamma \subset T$ be an embedded curve. Its image $F_d(\gamma)$ under the Lagrangian push-off is
called the \emph{Lagrangian push-off} of $\gamma$. These curves are used to parametrize the Luttinger surgery.\\

A \emph{meridian} of $T$ is a curve isotopic to $\{t\}\times \partial D^2 \subset \partial(nbd(T))$ and it is denoted by $\mu_t$. Consider two
embedded curves in $T$ which intersect transversally in one point and consider their Lagrangian push-offs $m_T$ and $l_T$. The group
$H_1(\partial(nbd(T)) = H_1(T^3)$ is generated by $\mu_T, m_t$ and $l_T$. We take advantage of the commutativity of $\pi_1(T^3)$ and choose a
basepoint $t$ on $\partial(nbh(T))$, so that we can refer unambiguously to $\mu_T, m_T, l_T \in \pi_1(\partial(nbd(T)), t)$.\\

Under this notation, a general torus surgery is the process of removing a tubular neighborhood of $T$ in $M$ and glue it back in such a way that the
curve representing $\mu_T^km_T^pl_T^q$ bounds a disk for some triple of integers $k, p,$ and $q$. In order to obtain a symplectic manifold
after the surgery, we need to set $k = \pm 1$ (cf. \cite{[BK3]}).\\

When the base point $x$ of $M$ is chosen off the boundary of the tubular neighborhood of $T$, the based loops $\mu_T, m_t$ and $l_T$ are to be
joined by the same path in $M - T$. By doing so, these curves define elements of $\pi_1(M - T, x)$. The 4-manifold $Y$ resulting from Luttinger surgery on $M$ has fundamental group

\begin{center}
$\pi_1(M - T) / N(\mu_T m^p_T l^q_T)$\\
\end{center}

where $N(\mu_T m^p_T l^q_T)$ denotes the normal subgroup generated by $\mu_T m^p_T l^q_T$.\\

We proceed now with the fundamental group calculations needed to prove Theorem 1. To do so, we plug into the previous general picture the information we have for the
telescoping triples. Let $(X, T_1, T_2)$ be a telescoping triple. The fundamental group of $X$ has the presentation $<t_1, t_2| [t_1, t_2] =
1>$. Let us apply $+ 1/p$ Luttinger surgery on $T_1$ along $l_{T_1}$ and call $Y_1$ the resulting manifold. Since the meridian $\mu_{T_1}$ is
trivial we have\\

\begin{center}
$\pi_1(Y_1) = \pi_1(X - T) / N(\mu_T m^0_{T_1} l_{T_1}^p) = \Z \oplus \Z / N(1 \cdot 1 \cdot l_{T_1}^p)$.\\
\end{center}

Thus, $\pi_1(Y_1) = <t_1, t_2| [t_1, t_2] = 1, t_2^p = 1>$.\\

Let us apply now $+ 1/q$ Luttinger surgery on $T_2$ along $m_{T_2}$ and call the resulting manifold $Y_2$ the resulting manifold. 
Since the meridian $\mu_{T_2}$ is
trivial we have

\begin{center}
$\pi_1(Y_2) = \Z \oplus \Z_p / N(1\cdot m^q_{T_2} \cdot 1) $.\\
\end{center}

Thus, $\pi_1(Y_1) = <t_1, t_2| [t_1, t_2] = 1, t_1^q = 1 = t_2^p>$.\\

The reader might have already noticed the symmetry of these calculations.

\begin{proposition} Let $(X, T_1, T_2)$ be a minimal telescoping triple. Let $l_{T_1}$ be a Lagrangian push off of a curve on $T_1$ 
and $m_{T_2}$ the Lagrangian push off of a curve on $T_2$ so that $l_{T_1}$ and $m_{T_2}$ generate $\pi_1(X)$.
\begin{itemize}
\item The minimal symplectic 4-manifold obtained by performing either $+1/p$ Luttinger surgery on $T_1$ along $l_{T_1}$ or $+1/p$ surgery on 
$T_2$ along $m_{T_2}$ has fundamental group isomorphic to $\Z \oplus \Z_p$.\\
\item The minimal symplectic 4-manifold obtained by performing $+1/p$ Luttinger surgery on $T_1$ along $l_{T_1}$ and
$+1/q$ surgery on $T_2$ along $m_{T_2}$ has fundamental group isomorphic to $\Z_q \oplus \Z_p$.\\
\end{itemize}
\end{proposition}

The proof is ommited. It is based on a repeated use of Lemma 2 in \cite{[BK3]} and Usher's theorem (cf. \cite{[Ush]}). The reader is
suggested to look at the proofs of theorems 8, 10 and 13 of \cite{[BK3]} for a blueprint to the proof.\\

%
%

\section{Proof of Theorem 1}

\begin{proof} The possible choices for characteristic numbers in Theorem 1 are in a one-to-one correspondence with the telescoping triples
of Proposition 8. The enumeration indicates that, in order to produce the manifold in Theorem 1 with one of the choices for characteristic numbers claimed in
item $\#$ (k), we start with the telescoping triple of item $\#$ (k) in Proposition 8 ($k\in \{1, 2, 3, 4, 5, \ldots, 14,
15\}$). Let $S:= (X, T_1, T_2)$ be the chosen minimal telescoping
triple. The manifolds of Theorem 1 are produced by applying Luttinger surgery to S according to the choice of
characteristic numbers. By Proposition 11 we know that out of $S$ one produces two symplectic manifolds: $Y_1$ with $\pi_1 = \Z \oplus \Z_p$
and $Y_2$ with $\pi_1 = \Z_q \oplus \Z_p$. Since Luttinger surgery does not change the Euler characteristic nor the
signature, the resulting manifolds $Y_1$ and $Y_2$ share the same characteristic numbers as $X$.\\

Proposition 11 states that $Y_1$ and $Y_2$ are minimal. By Hamilton-Kotschick result, both of them are irreducible. The calculation of the
characteristic numbers of $Y_1$ and $Y_2$ is straight-forward. Since our chosen $S$ was arbitrary, this finishes the proof.\\
\end{proof}

\section{Exotic Smooth Structures on 4-Manifolds with Abelian Finite Non-Cylic $\pi_1$}

The purpose of this section is to put on display the exotic smooth structures for the manufactured manifolds having $\pi_1 =
\Z_p\oplus \Z_p$, i.e., to prove Proposition 2.\\

\subsection{Smooth Topological Prototype}

We proceed to construct the underlying smooth manifold on which infinitely many exotic smooth structures will be
displayed. Start with the product of a Lens space and a circle: $L(p, 1)\times S^1$. Its Euler characteristic is zero 
as well as its signature. Consider the map\\

\begin{center}
$L(p, 1)\times S^1 \rightarrow L(p, 1)\times S^1$\\

$\{pt\} \times \alpha \mapsto \{pt\} \times \alpha^p$\\
\end{center}

We perform surgery on $L(p, 1)\times S^1$: cut out the loop $\alpha^p$ and glue in a disc in order to kill the
corresponding generator\\
\begin{center}
$\widetilde{L(p, 1)\times S^1} := L(p, 1)\times S^1 - (S^1\times D^3) \cup S^2\times D^2$.\\
\end{center}

The resulting manifold has zero signature and Euler characteristic two. By the Seifert-Van Kampen theorem, one concludes
$\pi_1(\widetilde{L(p, 1)\times S^1}) = \Z_p \oplus \Z_p$. \\

Since we are aiming at non-spin manifolds, our topological prototypes will have the shape\\

\begin{center}
$b_2^+\mathbb{CP}^2 \# b_2^-\overline{\mathbb{CP}}^2 \# \widetilde{\L(p, 1)\times S^1}$\\
\end{center}

but spin 4-manifolds with $\pi_1 = \Z_p \oplus \Z_p$ are also built in such a straight-forward manner.\\

\subsection{An infinite family $\{X_n\}$} We apply now the procedure described in \cite{[FS]} and \cite{[FPS]} to
produce infinitely many distinct smooth structures on any of our topological prototypes. Let $X_0$ be the manifold obtained by applying $+1/p$
Luttinger surgery on $T_2$ along $l_{T_2}$ to any of the manifolds from the telescoping triples previously
constructed. Since $X_0$ is a minimal symplectic manifold with $b_2^+ = 2$, its Seiberg-Witten invariant is 
non-trivial by \cite{[Ta1]}.\\

The infinite family $\{X_n\}$ is obtained by applying a $+ n/p$ torus surgery to $X_0$ on $T_1$ along $m_{T_1}$.
Notice that now $k = n$ acording to our notation of Section 4; only the case $k = 1 = n$ produces a symplectic
manifold.  We take a closer look at the process to see that we comply with the hypothesis of Corollary 2 in
\cite{[FPS]}.\\

The boundary of the tubular neighborhood of $T_1$ in $X_0$ is a 3-torus whose fundamental group is generated by the
loops $\mu_{T_1}, m_{T_1}$ and $l_{T_1}$. Notice that in $\pi_1(X_0 - T_1)$, the meridian is trivial 
$\mu_{T_1} = 1$, $m_{T_1} = x$ and $l_{T_1} = 1$, where $x$ is a generator in $\pi_1(X_0) = \Z_p \oplus \Z x$. The manifolds in the
family $\{X_n\}$ can be described as the result of applying to $X_0$ a $n/p$ surgery on $T_1$ along $m_{T_1}$, and so
$\mu_{T_1}^nm_{T_1} = x^p$ is killed.\\

Let $X$ be the manifold obtained from $X_0 - T_1$ by gluing a thick torus $T^2\times D^2$ in a manner that $\gamma
= S^1\times \{1\} \times \{1\}$ is sent to $l_{T_1}$, $\lambda = \{1\}\times S^1 \times \{1\}$ is sent to
$\mu_{T_1}$, and $\mu_X = \{(1, 1)\}\times \partial D^2$ is sent to $m_{T_1}^{-p}$. If $n \neq 1$, the manifold
$X$ will not be symplectic, but in any case $\pi_1(X) = \Z_p\oplus \Z_p$. Denote by $\Lambda \subset X$ the core torus 
of the surgery.\\

Notice that given the identifications on the loops during the surgery, $\lambda = \mu_{T_2} = 1$, thus it is nullhomotopic in $X_0 - T_1 = X - \Lambda$; in particular,
$\lambda$ is nullhomologous. The torus surgery kills one generator of $H_1$ and two generators of $H_2$; $\Lambda$ 
is a nullhomologous torus. One obtains a manifold $X_n$ by applying 1/n surgery on $\Lambda$ along $\lambda$ with 
$\pi_1(X_n) = \Z_p \oplus \Z_p$. The manifold $X_0$ can be recovered from $X$ by applying a $0/1$ surgery on $\Lambda$ along $\lambda$.\\

By Corollary 2 in \cite{[FPS]}, we produce an infinite family $\{X_n\}$ of pairwise non-diffeomorphic 4-manifolds. These manifolds
will have the same cohomology ring as the corresponding topological prototype. Thus we have the following lemma.

\begin{lemma} There exists an infinite family $\{X_n\}$ of pairwise non-diffeomorphic irreducible non-symplectic 
4-manifolds with $\pi_1 = \Z_p \oplus \Z_p$ sharing the same Euler characteristic, signature and type as a given
topological prototype constructed in the previous subsection.
\end{lemma}


\subsection{Homeomorphism Criteria} Now we need to see that the manifolds produced share indeed the same underlying
topological prototype. Ian Hambleton and Matthias Kreck proved the needed homeomorphism criteria in \cite{[HK93]} (theorem B). They showed that 
topological 4-manifolds with odd order fundamental group and large Euler characteristic are classified up to
homeomorphism by explicit invariants.\\

The precise statement of their result includes a lower bound for the Euler 
characteristic in terms of an integer number $d(\pi)$, which depends on the fundamental group of the manifold. We proceed to explain the 
notation employed.\\

Let $\pi_1 = \pi$ be a finite group and let $d(\pi)$ be the minimal $\Z$-rank for the abelian group $\Omega^3 \Z \otimes_{\Z[\pi]} \Z$. One minimizes over all representatives 
of $\Omega^3\Z$, the kernel of a projective resolution of length three (cf. \cite{[HK]}) of $\Z$ over the group ring $\Z[\pi]$. In
particular, $\Omega^3 \Z$ is a submodule of $\pi_2(X)$. The
minimal representative is given by $\pi_2(K)$, where $K$ is a two-complex with the given $\pi_1$.\\

The result we will use in order to conclude on the homeomorphism type of our manifolds is the following\\

\begin{theorem} (Hambleton-Kreck, cf \cite{[HK93]}). Let $M$ be a closed oriented manifold of dimension four, and let $\pi_1(X) = \pi$ be a finite group of
odd order. When $\omega_2(\tilde{X}) = 0$ (resp. $\omega_2(\tilde{X}) \neq 0$), assume that
\begin{center}
$b_2(X) - |\sigma(X)| > 2 d(\pi)$,
\end{center}
$(resp.  >2d(\pi) + 2)$. Then $M$ is classified up to homeomorphism by the signature, Euler characteristic, type,
Kirby-Siebenmann invariant, and fundamental class in $H_4(\pi, \Z) / Out(\pi)$.\\
\end{theorem}

Notice that since $p\geq 3$ is assumed to be a prime number, $\pi_1$ has odd order and no 2-torsion. Therefore, the type of the manifold is
indicated by the parity of its intersection form over $\Z$. All of our manufactured manifolds are non-spin; since they are smooth, the 
Kirby-Siebenmann invariant vanishes.\\ 

For the finite groups $\pi = \Z_p \oplus \Z_p$, we claim
\begin{center}
$d(\pi) = 1$.\\
\end{center}

We are indebted to Matthias Kreck for explaining us the argument \cite{[Kr]}. Assume $\pi = \pi_1$ is a finite group and let $K$ be a 2-complex
with fundamental group $\pi_1$. The minimal Euler characteristic of a $K$ is given by $d(\pi) + 1$. We claim $d(\pi) = 1$.\\

Consider the map from $K$ to the Eilenberg-MacLane space $K(\pi, 1)$ which induces an isomorphism on $\pi_1$. Then the induced map on $H_2(K; \Z_p)$
is surjective. Thus, the Euler characteristic of $K$ is greater or equal than 3 - 2 + 1. This implies $d(\pi)$ is greater or equal than 1.\\

To conclude now $d(\pi) = 1$, consider the standard presentation of $Z_p\oplus \Z_p$ given by
\begin{center}
$<x, y | x^p = 1, y^p = 1, [x, y] = 1>$.\\
\end{center}

The 2-complex realising this presentation has Euler characteristic $2 = d(\pi) + 1$. Therefore, $d(\pi) =1$ as claimed.\\

In order to conclude on the homeomorphism type of our manufactured 
manifolds, we only need to know the numerical 
invariants $b_2^+$ and $b_2^{-}$ which need to satisfy\\

\begin{center}
$b_2(X) - |\sigma(X)| > 4$.\\
\end{center}


\subsection{Proof of Proposition 2}

The proof of Proposition 2 is now clear if one rewrites it in the following form. Using , $n\geq 2$ if $m = 0$ or does not appear; $m\geq
2$ if $n = 0$\\

\begin{proposition} Assume $n + m \geq 2$. The manifolds
\begin{center}
$b_2 ^+\mathbb{CP}^2 \# b_2 ^{-} \overline{\mathbb{CP}}^2 \# \widetilde{L(p, 1)\times S^1}$ \\
\end{center}

with the following coordinates admit infinitely many exotic irreducible smooth structures, only one of which is symplectic.

\begin{enumerate}
\item $(b_2^+, b_2^-) = (2n - 1, 3n - 1)$,\\
\item $(b_2^+, b_2^-) = (2n - 1, 5n - 1)$,\\
\item $(b_2^+, b_2^-) = (2n - 1, 6n - 1)$,\\
\item $(b_2^+, b_2^-) = (2n - 1, 8n - 1)$,\\
\item $(b_2^+, b_2^-) = ((2 + 2g)n - 1, (4 +2g)n - 1)$,\\
\item $(b_2^+, b_2^-) = (2n + (2 + 2g)m - 1, 3n + (4 +2g)m - 1)$,\\
\item $(b_2^+, b_2^-) = (2n + 2m - 1, 3n + 5m - 1)$,\\
\item $(b_2^+, b_2^-) = (2n + 2m - 1, 3n + 6m - 1)$,\\
\item $(b_2^+, b_2^-) = (2n + 2m - 1, 3n + 8m - 1)$,\\
\item $(b_2^+, b_2^-) = ((2 + 2g)n + 2m - 1, (4 +2g)n + 5m - 1)$,\\
\item $(b_2^+, b_2^-) = ((2 + 2g)n + 2m - 1, (4 +2g)n + 6m - 1)$,\\
\item $(b_2^+, b_2^-) = ((2 + 2g)n + 2m - 1, (4 +2g)n + 8m - 1)$,\\
\item $(b_2^+, b_2^-) = (2n + 2m - 1, 5n + 6m - 1)$,\\
\item $(b_2^+, b_2^-) = (2n + 2m - 1, 5n + 8m - 1)$,\\
\item $(b_2^+, b_2^-) = (2n + 2m - 1, 6n + 8m - 1)$.\\
\end{enumerate}

\end{proposition}

\begin{proof} The infinite families are provided by Lemma 12. Choosing the topological prototype accordingly to the coordinates, by Theorem 12 and the
discussion that follows we conclude on the homeomorphism type. Notice that the enumeration of the coordinates presented in Proposition 14 correspond
exactly to the ones in Theorem 1.
\end{proof}

\end{document}